\documentclass [12pt]{article}
\usepackage{amssymb,amsmath,amsthm}
\usepackage{color}

\DeclareMathOperator{\spn}{span}

\usepackage[cp1250]{inputenc}
\usepackage[active]{srcltx}

\newtheorem{thm}{Theorem}[section]
\newtheorem{prop}[thm]{Proposition}
\newtheorem{lem}[thm]{Lemma}
 \newtheorem{rem}[thm]{Remark}
 \newtheorem{cor}[thm]{Corollary}
\newtheorem{defn}[thm]{Definition}

\DeclareMathOperator{\supp}{supp}
\DeclareMathOperator{\ext}{ext}
\newcommand{\C}{\mathbb{C}}
\newcommand{\R}{\mathbb{R}}

\newcommand{\la}{\lambda}

\long\def\comment#1{{}}
\def\ph{\phantom{-}}
\def\p{\ph}
 \title{Indeterminate Jacobi operators II}
 \author{Christian Berg and Ryszard Szwarc}

\begin{document}

 \maketitle
file: besz10v2.tex

\begin{abstract}  We consider the Jacobi operator $(T,D(T))$ associated with an indeterminate Hamburger moment problem, and present countable subsets $S$ of the domain $D(T)$  such that $\spn(S)$ is dense in $\ell^2$. As an example  we have $S=\{(p_n(u))+B(u)(p_n(0)) \mid D(u)=0, u\neq 0\}$, where $(p_n)$ denotes the orthonormal polynomials of the moment problem and $B, D$ are two of the Nevanlinna functions. 
It is also proved that sets like $S$ are optimal in the sense that if one vector is removed, then the span is no longer dense.
\end{abstract}



\section{Introduction and main results} This paper is a continuation of \cite{B:S5}  by improving certain results about the domain of the indeterminate Jacobi operator. We  consider the Jacobi matrix $J$ associated with a moment sequence $s=(s_n)_{n\ge 0}$ of the form
\begin{equation}\label{eq:mom}
s_n=\int x^n\,d\mu(x),\quad n=0,1,\ldots,
\end{equation}
where $\mu$ is a positive measure on $\R$  with infinite support and moments of every order. It is a tridiagonal matrix of the  form
\begin{equation}\label{eq:Jac}
J=\begin{pmatrix}
b_0 & a_0 & 0 & \hdots\\
a_0 & b_1 & a_1 & \hdots\\
0 & a_1 & b_2 & \hdots\\
\vdots &\vdots & \vdots & \ddots
\end{pmatrix},
\end{equation}
where $a_n>0, b_n\in\R, n\ge 0$  are given by the three term recurrence relation
\begin{equation*}\label{eq:3term}
xp_n(x)=a_np_{n+1}(x)+b_np_n(x)+a_{n-1}p_{n-1}(x),\quad n\ge 0, \quad a_{-1}:=0.
\end{equation*}
Here $(p_n)_{n\ge 0}$ is the sequence of orthonormal polynomials associated with $\mu$, hence satisfying
\begin{equation*}\label{eq:op}
\int p_n(x)p_m(x)\,d\mu(x)=\delta_{n,m}.
\end{equation*}
As usual $p_n$ is a real polynomial of degree $n$ with positive leading coefficient. As in \cite{B:S5} we follow the terminology of \cite{Sch}. Basic results about the classical moment problem can also be found in \cite{Ak} and \cite{S}. Recent  results about indeterminate moment problems can be found in \cite{B:S1}, \cite{B:S2}, \cite{B:S3}, \cite{B:S4}, \cite{B:S5}.

 It is easy to see that the proportional measures $\la\mu, \la>0$ lead to the same Jacobi matrix $J$, and the well-known Theorem of Favard (see \cite[Theorem 5.14]{Sch}) states that any matrix of the form \eqref{eq:Jac} with $a_n>0, b_n\in\R$ comes from a unique  moment sequence  $(s_n)$ as above, normalized such that $s_0=1$. In the following we shall always assume that this normalization holds, and consequently the solutions $\mu$ of \eqref{eq:mom} are probability measures  and $p_0=1$.

The Jacobi matrix acts as a symmetric operator  in the Hilbert space $\ell^2$ of square summable complex sequences. Its domain $\mathcal F$ consists of the complex sequences $(c_n)_{n\ge 0}$ with only finitely many non-zero terms, and the action is multiplication of the matrix $J$ by $c\in\mathcal F$ considered as a column, i.e.,
\begin{equation}\label{eq:J} 
(Jc)_n:=a_{n-1}c_{n-1}+b_n c_n+a_n c_{n+1},\quad n\ge 0.
\end{equation}
Denoting $(e_n)_{n\ge 0}$ the standard orthonormal basis of $\ell^2$, we have
\begin{equation*}\label{eq:F}
\mathcal F=\spn \{e_n|n\ge 0\}.
\end{equation*}

\begin{defn} The Jacobi operator associated with $J$ is by definition the closure $(T,D(T))$ of the symmetric operator $(J,\mathcal F)$.
\end{defn}

It is a classical fact that the closed symmetric operator $(T,D(T))$ has deficiency indices  either $(0,0)$  or $(1,1)$. These cases occur precisely if the  moment sequence \eqref{eq:mom} is {\it determinate} or {\it indeterminate}, i.e., there is exactly one or several solutions $\mu$ satisfying \eqref{eq:mom}.   

 By definition $D(T)$ consists of those $c\in\ell^2$ for which there exists a sequence
 $(c^{(k)})\in\mathcal F$ such that $\lim_{k\to\infty}c^{(k)}=c$ and $(Jc^{(k)})$ is a convergent sequence in $\ell^2$. For  such $c$ we have $Tc=\lim_{k\to\infty} Jc^{(k)}$, and this limit is independent of the choice of approximating sequence $(c^{(k)})$. 
 
 Clearly, $D(T)$ is closed under complex conjugation and 
 $$
 T\overline{c}=\overline{Tc},\quad c\in D(T).
 $$
 
 We recall that the adjoint operator $(T^*,D(T^*))$ is the maximal operator associated with $J$, cf. \cite[Proposition 6.5]{Sch}. In fact, the matrix product of $J$  and any column vector $c$ makes sense, cf. \eqref{eq:J}, and $D(T^*)$ consists of those $c\in\ell^2$ for which the product $Jc$ belongs to $\ell^2$. For $c\in D(T^*)$ we have $T^*c=Jc$.
 
In this paper we will only consider the  indeterminate case of the Jacobi operator $(T,D(T))$, where it is known that  the set of solutions $\mu$ to \eqref{eq:mom} is an infinite convex set $V$. The polynomials of the second kind $(q_n)$ are given as
\begin{equation*}
q_n(z)=\int \frac{p_n(z)-p_n(x)}{z-x}\,d\mu(x), \quad z\in \C,
\end{equation*}
where $\mu\in V$ is arbitrary. Note that $q_n$ as well as $p_n$ are independent of $\mu\in V$ in the indeterminate case.

We define  the sequences
\begin{equation}\label{eq:frak}
 \mathfrak{p}_z:=(p_n(z)), \mathfrak{q}_z:=(q_n(z)),\quad z\in\C,
 \end{equation}
 where we have followed the terminology of \cite{Sch}. It is known that they belong to
$\ell^2$ because of indeterminacy, and
   $||\mathfrak{p}_z||$ and $||\mathfrak{q}_z||$ are positive continuous functions for $z\in\C$.  It is therefore possible 
for $c\in\ell^2$ to define entire functions $F_c, G_c$ as
\begin{equation}\label{eq:FGc}
F_c(z)=\sum_{n=0}^\infty c_np_n(z),\quad G_c(z)=\sum_{n=0}^\infty c_nq_n(z),\quad z\in\C.
\end{equation}

It is well known that $F_c\in L^2(\mu)$ for any solution $\mu\in V$ and that
$$
\lim_{n\to\infty} \sum_{k=0}^n c_kp_k(z)= F_c(z)
$$
locally uniformly in $z\in\C$ and in $L^2(\mu)$ for any $\mu\in V$. Furthermore, Parseval's equation holds
\begin{equation}\label{eq:Pars}
\int |F_c(x)|^2\,d\mu(x)=||c||^2,\quad c\in\ell^2, \mu\in V.
\end{equation}

We recall the following four entire functions of two complex variables, called the {\it Nevanlinna functions} of the indeterminate moment problem:
\begin{eqnarray}
A(u,v)&=&(u-v)\sum_{k=0}^\infty q_k(u)q_k(v)\label{eq:A},\\
B(u,v)&=&-1+(u-v)\sum_{k=0}^\infty p_k(u)q_k(v) \label{eq:B},\\
C(u,v)&=&1+(u-v)\sum_{k=0}^\infty q_k(u)p_k(v)\label{eq:C},\\
D(u,v)&=&(u-v)\sum_{k=0}^\infty p_k(u)p_k(v)\label{eq:D},
\end{eqnarray}
see Section 7.1 in \cite{Sch}. They satisfy the  fundamental determinant equation
\begin{equation}\label{eq:2vardet}
A(u,v)D(u,v)-B(u,v)C(u,v)=1, \quad u,v\in\C.
\end{equation}
 
We define entire functions of one variable by setting the second variable to 0, i.e., 
\begin{equation}\label{eq:A-D}
A(u)=A(u,0),\;B(u)=B(u,0),\;C(u)=C(u,0),\;D(u)=D(u,0),
\end{equation}
and \eqref{eq:2vardet} becomes

\begin{equation}\label{eq:det}
A(u)D(u)-B(u)C(u)=1,\quad u\in\C.
\end{equation}

By Section 6.5 in \cite{Sch} we have
\begin{equation}\label{eq:p,q}
\mathfrak{p}_z, \mathfrak{q}_z \in D(T^*),\quad T^*\mathfrak{p}_z=z\mathfrak{p}_z, 
\quad T^*\mathfrak{q}_z=e_0+z\mathfrak{q}_z,\quad z\in\C.
\end{equation}

A main result of \cite{B:S5} states the following:

\begin{thm}\label{thm:pqmain} For all $z\in\C$ we have $\mathfrak{p}_z, \mathfrak{q}_z\notin D(T)$.

Let $u,v\in\C$ be given.
\begin{enumerate}
\item[(i)] There exists $\alpha\in\C$ such that $\mathfrak{p}_u+\alpha \mathfrak{p}_v\in D(T)$ if and only if $D(u,v)=0$. In the affirmative case $\alpha$  is uniquely determined as
$\alpha=B(u,v)$. 

\item[(ii)] There exists $\beta\in\C$ such that $\mathfrak{q}_u+\beta \mathfrak{q}_v\in D(T)$ if and only if $A(u,v)=0$. In the affirmative case $\beta$  is uniquely determined as $\beta=-C(u,v)$.

\item[(iii)] There exists $\gamma\in\C$ such that $\mathfrak{p}_u+\gamma \mathfrak{q}_v\in D(T)$ if and only if $B(u,v)=0$. In the affirmative case $\gamma$  is uniquely determined as $\gamma=-D(u,v)$. In particular $\mathfrak{p}_u+\gamma\mathfrak{q}_u\notin D(T)$ for all $u, \gamma\in\C$.
\end{enumerate}
\end{thm}

From  this theorem we have the following concrete subspaces of $D(T)$:
\begin{eqnarray*}
P&=&\spn\{\mathfrak{p}_u+B(u,v){\mathfrak{p}}_v \mid u,v\in\C, D(u,v)=0, u\neq v\},\\
Q&=&\spn\{\mathfrak{q}_u-C(u,v){\mathfrak{q}}_v \mid u,v\in\C, A(u,v)=0, u\neq v\},\\
M&=&\spn\{\mathfrak{p}_u-D(u,v){\mathfrak{q}}_v \mid u,v\in\C, B(u,v)=0\}.
\end{eqnarray*}
Note that $\mathfrak{p}_u+B(u,v)\mathfrak{p}_{v} =\mathfrak{q}_u-C(u,v)\mathfrak{q}_v=0$  for $u=v$ and that $Q\subseteq e_0^\perp=\{(c_n)_{n\ge 0}\in\ell^2 \mid c_0=0\}$.

For a fixed number $v_0\in\R$ we define the following subspaces of $P,Q,M$
respectively
\begin{eqnarray}
P(v_0)&=&\spn\{\mathfrak{p}_u+B(u,v_0)\mathfrak{p}_{v_0} \mid  D(u,v_0)=0, u\neq v_0\},\label{eq:P}\\
Q(v_0)&=&\spn\{\mathfrak{q}_u-C(u,v_0){\mathfrak{q}}_{v_0} \mid  A(u,v_0)=0, u\neq v_0\},\label{eq:Q}\\
M(v_0)&=&\spn\{\mathfrak{p}_u-D(u,v_0){\mathfrak{q}}_{v_0} \mid  B(u,v_0)=0\}\label{eq:M}.
\end{eqnarray}
In these definitions it is important to remember that if $F$ is any of the functions $A,B,C,D$ of two variables, then $\{u\in\C \mid F(u,v_0)=0\}$
is a countably infinite set of real numbers, cf. Theorem 1.3 in \cite{B:S5}.

Letting $\mu[v_0]$ denote the unique N-extremal measure in $V$ with $v_0\in\supp(\mu[v_0])$,
cf. Proposition~\ref{thm:supportNext}, we have by Theorem 3 in \cite{B:C}
\begin{equation}\label{eq:supp}
\supp(\mu[v_0])=\{u\in\R \mid D(u,v_0)=0\},
\end{equation}
and hence
\begin{equation} 
P(v_0)=\spn\{\mathfrak{p}_u+B(u,v_0)\mathfrak{p}_{v_0} \mid u\in\supp(\mu[v_0])
\setminus\{v_0\}\}.
\end{equation}

In the next section we recall the parametrization $\mu_t, t\in\R^*=\R\cup\{\infty\}$ of the N-extremal measures in $V$, and by Proposition~\ref{thm:supportNext} we have
\begin{equation}\label{eq:N-ext}
\mu[v_0]=\mu_t,\quad t=-B(v_0)/D(v_0)
\end{equation}
with the convention that $t=\infty$ if $D(v_0)=0$.

Also the set $\{u\in\R \mid B(u,v_0)=0\}$ is the support of an N-extremal measure. By formula (5.6) in \cite{B:S5} we have $B(u,v_0)=B(u)C(v_0)-D(u)A(v_0)$
so
\begin{equation}\label{eq:N-ext2}
\{u\in\R \mid B(u,v_0)=0\}=\supp(\mu_s),
\end{equation}
where $s=-A(v_0)/C(v_0)$ interpreted as $s=\infty$ if $C(v_0)=0$.

 The following result sharpens that $D(T)$ is dense in $\ell^2$ by giving concrete dense subspaces of $D(T)$, which are all spanned by countably many vectors. It is the main result of this paper, but we believe that the results of Proposition~\ref{thm:1-2} and Corollary~\ref{thm:limit} are new.

  \begin{thm}\label{thm:concrete} For each $v_0\in\R$ the subspaces $P(v_0)$ and $M(v_0)$ are dense in $\ell^2$, while $Q(v_0)$ is dense in $e_0^\perp$.
\end{thm}

\begin{rem}\label{thm:optimal1}
{\rm We claim that the subspaces $P(v_0), v_0\in \R$ are optimal as dense subspaces in $\ell^2$ in the following sense. Let $D(u_2,v_0)=0$ for $u_2\neq v_0$ and define
\begin{equation}\label{eq:P'} 
P'(v_0)=\spn\{\mathfrak{p}_u+B(u,v_0)\mathfrak{p}_{v_0} \mid D(u,v_0)=0, u\neq v_0,u_2\}.
\end{equation}
Then $P'(v_0)$ is not dense in $\ell^2$.

Similarly $Q(v_0), v_0\in\R$ are optimal as  dense subspaces of $e_0^\perp$ in the sense that if $A(u_2,v_0)=0, u_2\neq v_0$ then
\begin{equation}\label{eq:Q'} 
Q'(v_0)=\spn\{\mathfrak{q}_u-C(u,v_0)\mathfrak{q}_{v_0} \mid A(u,v_0)=0, u\neq v_0,u_2\}
\end{equation}
is not dense in $e_0^\perp$.
}
\end{rem}

\begin{rem}\label{thm:optimal2}
{\rm Let $u_1,u_2$ be different real numbers satisfying $B(u_1,v_0)=B(u_2,v_0)=0$ and define
\begin{eqnarray} 
M'(v_0)&=&\spn\{\mathfrak{p}_u-D(u,v_0)\mathfrak{q}_{v_0} \mid B(u,v_0)=0, u\neq u_1\},\label{eq:M'}\\
M''(v_0)&=&\spn\{\mathfrak{p}_u-D(u,v_0)\mathfrak{q}_{v_0} \mid B(u,v_0)=0, u\neq u_1, u_2\},\label{eq:M''}
\end{eqnarray}
then $M'(v_0)$ is dense in $\ell^2$ while $M''(v_0)$ is not dense in $\ell^2$.
}
\end{rem}

The proof of Theorem~\ref{thm:concrete} and  Remarks~\ref{thm:optimal1} and \ref{thm:optimal2}  will be given in Section 3.

\section{Preliminaries about indeterminate moment problems}

For the proof of Theorem~\ref{thm:concrete} we need the following polynomial approximations to the Nevanlinna functions.
 
\begin{prop}\cite[Proposition 5.24]{Sch}\label{thm:An-Dn} For $u,v\in\C$ and $n\ge 0$ we have
\begin{eqnarray*}
A_n(u,v)&:=&(u-v)\sum_{k=0}^nq_k(u)q_k(v)=
a_n\left|\begin{array}{cc}
 q_{n+1}(u)&\;q_{n+1}(v)\\q_{n}(u)&\;q_{n}(v)\end{array}\right|,\\
B_n(u,v)&:=&-1+(u-v)\sum_{k=0}^n p_k(u)q_k(v)=
a_n\left|\begin{array}{cc}
p_{n+1}(u)&\;q_{n+1}(v)\\p_{n}(u)&\;q_{n}(v)\end{array}\right|,\\
C_n(u,v)&:=&1+(u-v)\sum_{k=0}^n q_k(u)p_k(v)=
a_n\left|\begin{array}{cc}
q_{n+1}(u)&\;p_{n+1}(v)\\
q_{n}(u)&\;p_{n}(v)\end{array}\right|,\\
D_n(u,v)&:=&(u-v)\sum_{k=0}^np_k(u)p_k(v)=
a_n\left|\begin{array}{cc}
p_{n+1(}u)&\;p_{n+1}(v)\\
p_{n}(u)&\;p_{n}(v)\end{array}\right|.
\end{eqnarray*}
\end{prop}

The Jacobi operator $(T,D(T))$ has deficiency indices $(1,1)$ and the self-adjoint extensions in
$\ell^2$ can be parametrized as the operators $T_t, t\in\R^*=\R\cup \{\infty\}$ with domain  
\begin{equation}\label{eq:domTt}
D(T_t)=D(T)\oplus \C  (\mathfrak{q}_0 + t\mathfrak{p}_0) \;\mbox{for}\;t\in\R,\quad
D(T_\infty)=D(T)\oplus \C \mathfrak{p}_0
\end{equation}
and defined by the restriction of $T^*$ to the domain, cf. \cite[Theorem 6.23]{Sch}.
We recall that $\mathfrak{p}_0, \mathfrak{q}_0$ are defined in \eqref{eq:frak}.

For $t\in\R^*$ we define the solutions to the moment sequence \eqref{eq:mom}
\begin{equation}\label{eq:Next}
\mu_t(\cdot):=\langle E_t(\cdot)e_0,e_0\rangle,
\end{equation}
where $E_t(\cdot)$ is the spectral measure of the self-adjoint operator $T_t$.

The measures $\mu_t, t\in\R^*$ are precisely those measures $\mu\in V$  for which the polynomials $\C[x]$ are dense in $L^2(\mu)$ according to a famous theorem of M. Riesz, cf. \cite{Ri}.
 They are called N-extremal  in \cite{Ak} and von Neumann solutions in \cite{S}, and they form a compact subset of the set $\ext(V)$ of extreme points of the convex set $V$. However, $\ext(V)$ is known to be a dense subset of $V$.  The N-extremal measures are characterized by the formula
\begin{equation}\label{eq:Npar}
\int\frac{d\mu_t(x)}{x-z}=-\frac{A(z)+tC(z)}{B(z)+tD(z)},\quad z\in \C\setminus\R, t\in\R^*,
\end{equation}
where $A, B, C, D$ are the entire functions given in \eqref{eq:A-D}, cf. \cite[Theorem 7.6]{Sch}. Recall that \eqref{eq:det} holds, so the right-hand side of \eqref{eq:Npar} is a M\"{o}bius transformation in $t$. We note in passing that the solutions $\mu\in V$ different from the N-extremal ones are given in \eqref{eq:Npar}, when $t$ is replaced by a non-degenerate Pick function $\varphi:\C\setminus\R \to\C$, cf. Theorem 7.13 in \cite{Sch}.

We summarize some of the properties of $\mu_t$, which can be found in \cite{Ak} and \cite{Sch}.

\begin{prop}\label{thm:supportNext}
\begin{enumerate}
\item[(i)] The  solution $\mu_t$ is a discrete measure with support equal to the countable zero set $\Lambda_t$ of the entire function $B(z)+tD(z)$, with the convention that $\Lambda_\infty$ is the zero set of $D$. We have $\Lambda_t\subset\R$ for $t\in\R^*$ and
the mass of $\mu_t$ at $u\in\Lambda_t$ equals $||\mathfrak{p}_u||^{-2}$. 
\item[(ii)] The supports of two different N-extremal solutions are disjoint and interlacing. Each point $x_0\in\R$ belongs to the support of a unique N-extremal measure $\mu_t$, where
$t\in\R^*$ is given  as $t=-B(x_0)/D(x_0)$ if $D(x_0)\neq 0$ and $t=\infty$ if $D(x_0)=0$.
\item[(iii)] If $x_0\in\R$ belongs to the support of the N-extremal measure $\mu_t$, then the measure $\mu_t-\mu_t(x_0)\delta_{x_0}$ is determinate.
\end{enumerate}
\end{prop}

Putting $z=0$ in \eqref{eq:Npar}, which is possible when $0\notin \supp(\mu_t)$, leads to
\begin{equation}\label{eq:Npar0}
\int\frac{d\mu_t(x)}{x}=t,\quad t\in\R,
\end{equation}
because $A(0)=D(0)=0, B(0)=-1, C(0)=1$.

The following Proposition shows how to construct an orthonormal basis for $\ell^2$ from the support of an N-extremal measure. We give two proofs, different in style. 
 
\begin{prop}\label{thm:ONB} Let $(u_n)_{n\ge 1}$ be the points of the countable set $\Lambda_t$ supporting $\mu_t$, where $t\in\R^*$. The vectors $\mathfrak{p}_{u_n}/||\mathfrak{p}_{u_n}||, n\ge 1$ form an orthonormal basis for $\ell^2$.
\end{prop}

\begin{proof} First proof. We know that $B(u_n)+tD(u_n)=0$ for $n\ge 1$ if $t\in\R$ and that $D(u_n)=0$ for $n\ge 1$ if $t=\infty$. We then get the orthogonality of the sequence $(\mathfrak{p}_{u_n})$ by
\begin{eqnarray*}
&&(u_n-u_m)\langle \mathfrak{p}_{u_n}, \mathfrak{p}_{u_m}\rangle=D(u_n, u_m)
=B(u_n)D(u_m)-D(u_n)B(u_m)\\
&=&\left\{\begin{array}{ll}
 -tD(u_n)D(u_m)-D(u_n)(-tD(u_m))=0 & \mbox{if} \;\; t\in\R,\\
 0 & \mbox{if}\;\; t=\infty,
 \end{array}
  \right.
\end{eqnarray*}
cf. formula (5.8) in \cite{B:S5}.

Assume next that $c\in\ell^2$ satisfies $\langle c, \mathfrak{p}_{u_n}\rangle=0$ for $n\ge 1$. Then the function $F_c$ from \eqref{eq:FGc} vanishes on $\supp(\mu_t)$ and by \eqref{eq:Pars} we get $c=0$, showing that  $\spn\{\mathfrak{p}_{u_n}\mid n\ge 1\}$ is  dense in $\ell^2$.

Second proof. The mapping $U:L^2(\mu_t)\to \ell^2$ defined by
$$
Uf=\left(\langle f, p_k \rangle_{L^2(\mu_t)}\right)_{k\ge 0}
$$
is a linear isometry by the Parseval identity. As $Up_n=e_n$ the operator  $U$ is unitary.
The functions $||\mathfrak{p}_{u_n}||\delta_{u_n}$ ($\delta_{u_n}$ being 1 on $u_n$ and  0 on $u_m, m\neq n$) form an orthonormal basis in $L^2(\mu_t)$. Therefore
$$
U(||\mathfrak{p}_{u_n}||\delta_{u_n})=\mathfrak{p}_{u_n}/||\mathfrak{p}_{u_n}||
$$
form an orthonormal basis for $\ell^2$.
\end{proof}

\section{Proof of Theorem~\ref{thm:concrete}}
 In the proof of Theorem~\ref{thm:concrete} we need the following result about the functions $B(u,v), D(u,v)$. It is  of independent interest.

\begin{prop}\label{thm:1-2} (i) For $\mu\in V, v\in\C$ and any polynomial $p$ the functions $p(u)B(u,v), p(u)D(u,v)$ belong to $L^1(\mu)$ as functions of $u$ and
$$
\int p(u)B(u,v)\,d\mu(u)=\int p(u)D(u,v)\,d\mu(u)=0.
$$ 

(ii) For $v_0\in\R$ we have $B(u,v_0)\notin L^2(\mu)$ for any N-extremal measure $\mu$ different from the N-extremal measure $\mu_s$ supported by $\{u\in\R \mid B(u,v_0)=0\}$, cf. \eqref{eq:N-ext2}, while $B(u,v_0)$ is the zero element in $L^2(\mu_s)$.

(iii) For $v_0\in\R$ we have $D(u,v_0)\notin L^2(\mu)$ for any N-extremal measure $\mu\neq \mu[v_0]$, while $D(u,v_0)$ is the zero element in $L^2(\mu[v_0])$.
 \end{prop} 

\begin{proof} (i). With the notation of \eqref{eq:FGc} we have
\begin{eqnarray*}
p(u)B(u,v)&=&-p(u)+p(u)(u-v)F_{(q_k(v))}(u),\\
 p(u)D(u,v)&=&p(u)(u-v)F_{(p_k(v))}(u),
\end{eqnarray*}
so it is clear from the Cauchy-Schwarz inequality that $p(u)B(u,v), p(u)D(u,v)$ belong to $L^1(\mu)$ as functions of $u$ and furthermore that
$p(u)B_n(u,v)\to p(u)B(u,v)$ and $p(u)D_n(u,v)\to p(u)D(u,v)$  in $L^1(\mu)$ for $n\to\infty$. However,
\begin{equation}\label{eq:int0}
\int p(u)B_n(u,v)\,d\mu(u)=\int p(u)D_n(u,v)\,d\mu(u)=0
\end{equation}
 for $n>\deg(p)$ by Proposition~\ref{thm:An-Dn}, and (i) follows.

(ii). If $B(u,v_0)\in L^2(\mu)$ for an N-extremal measure $\mu$ it follows from \eqref{eq:int0} that $\int p(u)B(u,v_0)\,d\mu(u)=0$ for all polynomials $p$. Since the polynomials are dense in $L^2(\mu)$ we conclude that $B(u,v_0)=0$ for all $u\in\supp(\mu)$. Therefore $\mu$ must be the N-extremal measure supported by $\{u\in\R \mid B(u,v_0)=0\}$.

(iii). If $D(u,v_0)\in L^2(\mu)$ for an N-extremal measure $\mu$, we get that $D(u,v_0)=0$ for all $u\in\supp(\mu)$. This is not possible if $\mu\neq \mu[v_0]$ because then  
$\supp(\mu)$ is disjoint from $\supp(\mu[v_0])$ and this contradicts \eqref{eq:supp}, which also shows that  $D(u,v_0)$ is the zero element in $L^2(\mu[v_0])$.
\end{proof}

\begin{cor}\label{thm:limit} For any $v_0\in\R$ we have
\begin{eqnarray*}
&\lim_{n\to\infty} a_n^2[p_n(v_0)^2+p_{n+1}(v_0)^2]=\infty,\\
&\lim_{n\to\infty} a_n^2[q_n(v_0)^2+q_{n+1}(v_0)^2]=\infty,
\end{eqnarray*}
where $a_n$ is the entry $J_{n+1,n}$  of the Jacobi matrix \eqref{eq:Jac}.
\end{cor}

\begin{proof} To prove the first equation assume by contradiction that there is a subsequence $(n_k)$ such that $a_{n_k}^2[p_{n_k}(v_0)^2+p_{n_k+1}(v_0)^2]$ is bounded, say by $M$.  We know that $D_n(u,v_0)$ from Proposition~\ref{thm:An-Dn} tends locally uniformly in $u$ to $D(u,v_0)$ for $n\to\infty$, and hence for any $R>0$ and any $\mu\in V$ 
\begin{eqnarray*}
\int_{-R}^R D(u,v_0)^2\,d\mu(u)&=&\lim_{k\to\infty}\int_{-R}^R D_{n_k}(u,v_0)^2\,d\mu(u)\\
&\le&\limsup_{k\to\infty}\int_{-\infty}^\infty D_{n_k}(u,v_0)^2\,d\mu(u)\\
&=&\limsup_{k\to\infty} a_{n_k}^2[p_{n_k}(v_0)^2+p_{n_k+1}(v_0)^2]\le M.
\end{eqnarray*}
As $R>0$ is arbitrary, $D(u,v_0)$ belongs to $L^2(\mu)$ for any $\mu\in V$, but this is not true by Proposition~\ref{thm:1-2} (iii). 

The proof of the second equation is similar.
\end{proof}

The following Lemma will be crucial for the proof of Theorem~\ref{thm:concrete}.

\begin{lem}\label{thm:Hilbert} Let $(x_n)_{n\ge 1}$ denote an orthonormal basis in a Hilbert space $\mathcal H$, let $(a_n)_{n\ge 2}$ be a sequence of complex numbers and $\xi\in\mathcal H$ a vector satisfying $\langle \xi, x_1\rangle\neq 0.$

Then $\spn\{x_n+a_n\xi\mid n\ge 2\}$ is dense in $\mathcal H$ iff $\sum_{n=2}^\infty
|a_n|^2=\infty$.

Moreover, $\spn\{x_n+a_n\xi\mid n\ge 3\}$ is never dense in $\mathcal H$. 
\end{lem}

\begin{proof} 1. Assume first the existence of a vector $w\in\mathcal H, w\neq 0$ such that
$\langle w, x_n+a_n\xi\rangle=0$ for $n\ge 2$. Then we have
$\langle w,x_n\rangle=-\overline{a_n}\langle w,\xi\rangle$ for $n\ge 2$.
 
We claim that  $\langle w, \xi\rangle\neq 0$, for if $\langle w, \xi\rangle=0$, we get $\langle w, x_n\rangle=0$ for $n\ge 2$ so $w=\langle w, x_1\rangle x_1$, and hence
$$
0=\langle w, \xi\rangle=\langle w, x_1\rangle \langle x_1, \xi\rangle.
$$
By assumption $\langle \xi, x_1\rangle\neq 0$, so $\langle w, x_1\rangle =0$ and together with the equations $\langle w, x_n\rangle=0$ for $n\ge 2$, we get $w=0$, which is a contradiction. 

Using $\langle w, \xi\rangle\neq 0$, we get by Bessel's inequality
$$
||w||^2\ge \sum_{n=2}^\infty |\langle w, x_n\rangle|^2=|\langle w,\xi\rangle|^2\sum_{n=2}^\infty|a_n|^2,
$$
showing that $\sum_{n=2}^\infty|a_n|^2<\infty$.

2. Assume next that $\sum_{n=2}^\infty|a_n|^2<\infty$, and define $w\in\mathcal H$ by
$$
w=a_1x_1-\sum_{n=2}^\infty \overline{a_n}x_n
$$
with $a_1$ such that $\langle w, \xi\rangle=1$, i.e., such that
$$
a_1\langle x_1,\xi\rangle-\sum_{n=2}^\infty \overline{a_n}\langle x_n, \xi\rangle =1,
$$
which is possible as $\langle \xi, x_1\rangle\neq 0$.

For $n\ge 2$ we find
$\langle w, x_n+a_n\xi\rangle=-\overline{a_n}+\overline{a_n}\langle w, \xi\rangle=0$.
Note that  $w\neq 0$ because either $a_1\neq 0$ or if $a_1=0$, i.e., $\sum_{n=2}^\infty \overline{a_n}\langle x_n, \xi\rangle =-1$, then $a_n\neq 0$ for some $n\ge 2$.

3. To see that $\spn\{x_n+a_n\xi \mid n\ge 3\}$ is not dense in $\mathcal H$ we define
$y:=\langle x_1, \xi\rangle x_2-\langle x_2, \xi\rangle x_1$ and note that $y\neq 0$.
For $n\ge 3$ we find
\begin{eqnarray*}
&&\langle y, x_n+a_n\xi\rangle=\langle \langle x_1,\xi\rangle x_2-\langle x_2, \xi\rangle x_1, a_n\xi\rangle\\
&=& \overline{a_n}\left(\langle x_1, \xi\rangle \langle x_2, \xi\rangle-\langle x_2, \xi\rangle\langle x_1, \xi\rangle\right)=0.
\end{eqnarray*}
\end{proof}

\noindent{\it Proof of Theorem~\ref{thm:concrete} and Remarks~\ref{thm:optimal1} and \ref{thm:optimal2}} 

(i). For $v_0\in\R$ we shall see that the density of $P(v_0)$ in $\ell^2$ follows from
 Lemma~\ref{thm:Hilbert}. We write the numbers in $\{u\in\R \mid D(u,v_0)=0\}$ as a sequence $(u_n)_{n\ge 1}$ with $u_1=v_0$. According to Proposition~\ref{thm:ONB} the vectors $x_n=\mathfrak{p}_{u_n}/||\mathfrak{p}_{u_n}||, n\ge 1$ form an orthonormal basis for $\ell^2$. For $\xi=\mathfrak{p}_{v_0}$ we have $\langle \xi, x_1\rangle=||\mathfrak{p}_{v_0}||\neq 0$ and
$$
P(v_0)=\spn\{x_n+a_n\xi \mid n\ge 2\},\quad a_n=B(u_n,v_0)/||\mathfrak{p}_{u_n}||.
$$
Using Proposition~\ref{thm:supportNext} (i) we get
$$
\sum_{n=2}^\infty \frac{B(u_n,v_0)^2}{||\mathfrak{p}_{u_n}||^2}=\int_{\R\setminus\{v_0\}}B(u,v_0)^2\,d\mu[v_0](u),
$$  
which is infinite by Proposition~\ref{thm:1-2} (ii).

The subspace $P'(v_0)$ equals $\spn\{x_n+a_n\xi \mid n\ge 3\}$, so it is not dense in $\ell^2$ by Lemma~\ref{thm:Hilbert}.

(ii).  The case about $Q(v_0)$  can be deduced from case (i) by using the observation that
the polynomials $(q_{n+1}(x)/q_1(x))_{n\ge 0}$ are the orthonormal polynomials associated with the truncated Jacobi matrix $J^{(1)}$ obtained from $J$ by removing the first row and column. See \cite[p. 28]{Ak}, \cite[p. 131, p. 144]{B:S5} and \cite{Pe} for details.

(iii). We finally consider $M(v_0)$. From \eqref{eq:N-ext2}  we have $B(u,v_0)=0$ if and only if $u\in\supp(\mu_s)$, where 
$s=-A(v_0)/C(v_0)$, interpreted as $s=\infty$ if $C(v_0)=0$.
These zeros form a sequence denoted $(u_n)_{n\ge 1}$. The vectors $x_n=\mathfrak{p}_{u_n}/||\mathfrak{p}_{u_n}||, n\ge 1$ form an orthonormal basis for $\ell^2$, and with $\xi=\mathfrak{q}_{v_0}$ and $a_n=-D(u_n,v_0)/||\mathfrak{p}_{u_n}||, n\ge 1$  we have
$$
\langle \xi, x_1\rangle=1/((u_1-v_0)||\mathfrak{p}_{u_1}||)\neq 0,\quad M(v_0)=\spn\{x_n+a_n\xi \mid n\ge 1\}.
$$  
Furthermore,
$$
\sum_{n=1}^\infty a_n^2= \int D(u,v_0)^2\,d\mu_s(u)=\infty
$$
by Proposition~\ref{thm:1-2}(iii) since $\mu_s\neq \mu[v_0]$ because $B(v_0,v_0)=-1$.

It is clear that also $\sum_{n=2}^\infty a_n^2=\infty$, so by Lemma~\ref{thm:Hilbert} the subspace
$M'(v_0)=\spn\{x_n+a_n\xi \mid n\ge 2\}$ of $M(v_0)$ is dense in $\ell^2$, while
$M''(v_0)=\spn\{x_n+a_n\xi \mid n\ge 3\}$ is not dense in $\ell^2$. \hfill$\square$

\medskip
{\bf Acknowledgment} The authors want to thank two  independent referees for their careful comments to the manuscript.  This has led to improvements of the paper.

\noindent
Christian Berg\\
Department of Mathematical Sciences, University of Copenhagen\\
Universitetsparken 5, DK-2100 Copenhagen, Denmark\\
e-mail: {\tt{berg@math.ku.dk}}

\vspace{0.4cm}
\noindent
Ryszard Szwarc\\
Institute of Mathematics, University of Wroc{\l}aw\\
pl.\ Grunwaldzki 2/4, 50-384 Wroc{\l}aw, Poland\\ 
e-mail: {\tt{szwarc2@gmail.com}}

\end{document}